\documentclass[10pt]{article}\normalsize

\usepackage{epsfig}
\usepackage{epstopdf}
\usepackage[titletoc,title]{appendix}
\usepackage{graphicx}
\usepackage{latexsym, cite, bm, amsmath}
\usepackage{contour,soul,color}
\usepackage{booktabs,threeparttable} 
\usepackage[rotateright]{rotating} 
\usepackage{tabularx,multirow} 
\usepackage{cases}
\usepackage{float}
\usepackage{mathtools}
\graphicspath{{figure/}}
\usepackage{caption}
\usepackage{diagbox}
\usepackage{lscape}
\usepackage{amsfonts}
\usepackage{epstopdf}
\usepackage{amsmath,url,cases}
\usepackage{amsthm}

\usepackage{makecell}
\usepackage[english]{babel}
\usepackage[latin1]{inputenc}
\usepackage{times}
\usepackage{bbding}
\usepackage{latexsym}
\usepackage{epic}
\usepackage{indentfirst}
\usepackage{hyperref}
\usepackage{subfigure}
\usepackage{fancyvrb}
\usepackage{array}
\usepackage[linesnumbered,boxed,ruled,commentsnumbered,titlenotnumbered]{algorithm2e}
\usepackage{setspace}

\usepackage[left=1.10in,top=1.10in,right=1.10in,bottom=1.10in]{geometry}
\makeatletter
\newcommand{\nosemic}{\renewcommand{\@endalgocfline}{\relax}}
\newcommand{\dosemic}{\renewcommand{\@endalgocfline}{\algocf@endline}}
\newcommand{\pushline}{\Indp}
\newcommand{\popline}{\Indm\dosemic}
\let\oldnl\nl
\newcommand{\nonl}{\renewcommand{\nl}{\let\nl\oldnl}}
\makeatother

\numberwithin{equation}{section}

\newtheorem{theorem}{Theorem}[section]

\newtheorem{lemma}[theorem]{Lemma}

\newtheorem{remark}[theorem]{Remark}

\newtheorem{condition}{Condition}[section]

\numberwithin{equation}{section}

\def\be{\begin{equation}}
\def\ee{\end{equation}}

\def\la{\lambda}
\def \[{\begin{equation}}
\def \]{\end{equation}}

\def\R{{\mathbb R}}

\def\L{{\cal L}}

\def\S{{\cal S}}

\def\N{{\cal N}}

\def\st{\mathrm{s. t.}}

\def\nn{\nonumber}
\def\Tsf{\mathsf{T}}

\newcommand{\rev}[1]{{\color{black}{#1}}}

\begin{document}

\title{ An Alternating Direction Method of Multiplier\rev{s} with the BFGS \rev{U}pdate for Structured Convex Quadratic Optimization
}

\author{Yan Gu
and
Nobuo Yamashita
}
\date{\today}
\maketitle

\begin{abstract}
The alternating direction method of multipliers (ADMM) is an effective method for solving wide fields of convex problems. At each iteration, the classical ADMM solves two subproblems exactly. However, in many applications, it is expensive or impossible to obtain the exact solutions of the \rev{subproblems}. To overcome the difficulty, some proximal terms are added to the subproblems.
This class of methods normally solves the original subproblem approximately, and thus takes more iterations. This fact urges us to consider that a special proximal term can lead to a better result as the classical ADMM.
In this paper, we propose a proximal ADMM whose regularized matrix in the proximal term is generated by the BFGS update (or \rev{limited} memory BFGS) at every iteration. These types of matrices use second-order information of the objective function. The convergence of the proposed method is proved under certain assumptions. Numerical results are \rev{presented} to show the effectiveness of the proposed proximal ADMM.

\end{abstract}



\section{Introduction}
\label{intro}
We consider the following convex \rev{optimization} problem:
\[\label{convex}\begin{array}{ll}
 \hbox{minimize}&  \frac{1}{2}\|Ax-b\|^2+g(x)\\
 \hbox{subject to}&  x\in \R^n, \\ \end{array}\]
where  $g: \R^n \rightarrow \R\cup \{\infty\}$ is a proper convex function,  $A\in \R^{m \times n}$ and $b\in \R^m$. For example, ``$g$'' here can be an indicator function on a convex set or the $l_1$ penalty function \rev{defined as} $\|x\|_1:=\sum\nolimits_{i=1}^{m}|x_i|$. \rev{Problem} \eqref{convex} includes many important statistical learning problems such as the LASSO problem \cite{tibshirani1996regression}.
The number $n$ of variables in these learning problems is usually large.
\par
Let $f(x)=\frac{1}{2}\|Ax-b\|^2$.
Then problem \eqref{convex} can be written as
\[\label{cp} \begin{array}{ll}
 \hbox{minimize}&  f(x)+g(y)\\
 \hbox{subject to}&  x-y=0  \\
                  & x, y\in \R^n. \end{array}\]
Let $\L_{\beta}(x,y,\la)$ be the augmented Lagrangian function for \eqref{cp} defined by
\[\label{augL}
\L_{\beta}(x,y,\la):=f(x)+g(y)-\langle\la,x-y\rangle+\frac{\beta}{2}\|x-y\|^2,\]
where $\la\in\R^n$ is multipliers associated to the linear constraints and $\beta>0$ is a penalty parameter.

\rev{For solving problem \eqref{cp},}
the alternating direction method of multipliers (ADMM) was proposed by Gabay and Mercier \cite{gabay1976dual}, \rev{and} Glowinski and Marrocco \cite{glowinski1975approximation} in the mid-1970s. \rev{ADMM} generates sequence \rev{$\{(x^k, y^k, \lambda^k)\}$} via the following \rev{recursions}:
\vspace{-2mm}
\begin{subnumcases}{\label{equ:ADMM}}
\label{alx}x^{k+1}=\arg\min_{x}\L_{\beta}(x,y^k,\la^k),\\
\label{aly}y^{k+1}=\arg\min_{y}\L_{\beta}(x^{k+1},y,\la^k),\\
\label{all}\la^{k+1}=\la^k-\beta(x^{k+1}-y^{k+1}).
\end{subnumcases}
The global convergence of the ADMM \eqref{alx}-\eqref{all} can be established under very mild conditions \cite{boyd2011distributed}.

By noting the fact that the subproblem in \eqref{alx}-\eqref{all} may be difficult to solve exactly in many applications, Eckstein \cite{eckstein1994some} and He et al. \cite{HLHY2002} have considered to add proximal terms to the subproblems for different \rev{purposes}. Recently, Fazel et al. \cite{fazel2013hankel} proposed the following semi-proximal ADMM scheme:
\vspace{-1mm}
\begin{subnumcases}{\label{equ:sADMM}}
\label{equ:sADMM1} x^{k+1}=\arg\min_{x}\ \L_{\beta}(x,y^k,\la^k) + \frac12 \|x - x^{k}\|_{T}^2,\\
\label{equ:sADMM2} y^{k+1}=\arg\min_{y}\ \L_{\beta}(x^{k+1},y,\la^k) + \frac12 \|y - y^{k}\|_{S}^2,\\
\label{equ:sADMM3} \la^{k+1}=\la^{k}-\rev{\alpha}\beta(x^{k+1}-y^{k+1}),
\end{subnumcases}
where $\rev{\alpha} \in (0, (1+\sqrt 5)/2)$, $\|z\|_G = \sqrt {z^ \top Gz}$ for $z \in \R^n$ and $G \in \R^{n \times n}$. Fazel et al. \cite{fazel2013hankel} \rev{showed} its global convergence when $T$ and $S$ are positive semidefinite, which makes the algorithm more flexible. See \cite{deng2012global, he2012convergence, fazel2013hankel, xu2011class} for a brief history of the development of the semi-proximal ADMM and the corresponding convergence results. \rev{Quite recently, the papers \cite{li2016majorized, he2017optimal} relaxed the positive semidefinite requirement of the proximal matrix $T$ to be indefinite in some extend, and showed its global convergence.}

In this paper, we suppose that $y^{k+1}$ in \eqref{equ:sADMM2} is easily obtained. For example, if $g(y)=\tau \|y\|_1$ with $\tau>0$ and $S=0$, then $y^{k+1}$ \rev{is calculated} within $O(n)$. Then our main focus is how to solve \eqref{equ:sADMM1} when $n$ is large.
We may choose a reasonable positive semidefinite matrix $T$ so that we get $x^{k+1}$ quickly.

One of such examples of $T$ is $T = \rev{\xi} I -A^\top A$ with $\rev{\xi}> \lambda_{\max}(A^\top A)$, \rev{where $\lambda_{\max}(A^\top A)$ denotes the maximum eigenvalue of $A^\top A$.} Then \eqref{equ:sADMM1} is written as
\vspace{-1mm}
\begin{eqnarray*}
x^{k+1} &=&\arg\min_{x} \left\{ f(x) -\langle \lambda^k, x - y^k\rangle +
\frac{\beta}{2}\|x-y^{k}\|^2+\frac{1}{2}\|x-x^k\|_{T}^2  \right\} \\
&=& \arg\min_{x} \left\{\rev{\langle Ax^{k}-b, Ax \rangle - \langle \lambda^{k}, x\rangle} + \frac{\beta}{2}\|x-y^{k}\|^2+\frac{\rev{\xi}}{2}\|x-x^{k}\|^2 \right\}\\
&=& (\lambda^k + \beta y^k + \rev{\xi} x^k - A^\top A x^k+ A^\top b)/ (\beta + \rev{\xi}).
\end{eqnarray*}
The other example is $T = \rev{\xi} I - \beta I - A^\top A$ with $\rev{\xi}> \lambda_{\max}(\beta I + A^\top A)$. Then \eqref{equ:sADMM1} is written as
\begin{eqnarray*}
x^{k+1}= x^k - \rev{\xi}^{-1}(A^\top A x^k - A^\top b - \lambda^k + \beta x^k - \beta y^k ).
\end{eqnarray*}
In both cases $x^{k+1}$ is calculated within $O(mn)$.
However, since \rev{these} subproblems do not include second-order information on $f$, the convergence of ADMM with such $T$ might be slow.\par

\rev{We want a matrix $T$ to be the one} such that it is positive semidefinite, subproblem \eqref{equ:sADMM1} is easily solved, and it has some \rev{second-order} information on $f$. Let $M$ be the Hessian matrix of the augmented Lagrangian function $\L_\beta$, that is $M \colon = \nabla^2_{xx} \L_{\beta}(x,y,\la) =  A^\top A + \beta I$. Note that $M \succ 0$ whenever $\beta > 0$.
Then, we consider a matrix $B$ that has the following three properties:
\begin{itemize}
{\item[(i)] $T=B - M$;
 \item[(ii)] $B \succeq M $;
 \item[(iii)] $B$ has some second order information on $M$.
}
\end{itemize}
\rev{Note that properties (i) and (ii) imply that $T$ is positive semidefinite. Moreover, subproblem} \eqref{equ:sADMM1} is written as
\begin{eqnarray*}
x^{k+1}&=&\arg\min_{x} \left\{ f(x) -\langle \lambda^k, x - y^k\rangle +
\frac{\beta}{2}\|x-y^{k}\|^2+\frac{1}{2}\|x-x^k\|_{T}^2  \right\} \\
&=&\arg\min_{x} \left\{ \langle A^\top(Ax^{k}-b)+\beta(x^{k}-y^k)-\lambda^{k}, x \rangle
+\frac{1}{2}\|x-x^{k}\|_{B}^2  \right\} \\
&=&x^{k} - B^{-1}\left(A^\top A x^k - A^\top b - \lambda^k + \beta (x^k - y^k) \right).
\end{eqnarray*}

In this paper, we propose to construct \rev{$B^{-1}$} via the BFGS update at every iteration. \rev{Then subproblem \eqref{equ:sADMM1} is easily solved. Note that matrices} $B$ and $T$ at every step depend on $k$, that is, they become $B_k$ and $T_k$, and the resulting ADMM is a variable metric semi-proximal ADMM (short by {\color{blue}VMSP-ADMM}) given as:
\begin{subnumcases}{\label{equ:ADMMB}}
\label{equ:ADB1} x^{k+1}=\arg\min_{x}\ \L_{\beta}(x,y^k,\la^k) + \frac12 \|x - x^{k}\|_{T_k}^2,\\
\label{equ:ADB2} y^{k+1}=\arg\min_{y}\ \L_{\beta}(x^{k+1},y,\la^{k}) + \frac12 \|y - y^{k}\|_{S}^2,\\
\label{equ:ADB3} \la^{k+1}=\la^{k}-\beta(x^{k+1}-y^{k+1}).
\end{subnumcases}

\rev{
VMSP-ADMM is studied in \cite{HLHY2002} where the $T_k$ is assumed to be positive definite. The convergence and complexity results have been studied in \cite{lotito2009class, banert2016fixing, Goncalves2018}.
Moreover it is also closely related to the inexact ADMM, where the subproblems in \eqref{equ:ADMM} or \eqref{equ:sADMM} to be solved approximately with certain implementable criteria \cite{eckstein1992douglas, chen1994proximal, HLHY2002, yuan2005improvement, eckstein2017approximate, eckstein2018relative}.
In this paper, we suppose that subproblems \eqref{equ:ADB1}-\eqref{equ:ADB2} are solved exactly.
}

\medskip
The main contributions of the paper are as follows:
\begin{itemize}
{
\rev{
\item[1.] We propose an update formula on positive semidefinite matrices $T_k$ and $B_k$ via the BFGS update that satisfy the above three properties (i)-(iii).
}
\item[2.] We report some numerical results for the proposed methods which shows that they outperform the existing ADMM when $n$ \rev{and $m$ are} large.
}
\end{itemize}

\rev{
The rest of the paper is organized as follows. In Section \ref{section:convergence}, we propose a new ADMM with the BFGS update, and show its global convergence. In Section \ref{section:numerical}, we present some numerical experiment results for ADMM with the BFGS and limited memory BFGS update. Finally, we make some concluding remarks in Section \ref{section:conclusions}.
}

\medskip
\vspace*{1 ex}
\leftline{$\mathbf{Notations:}$~~Here we give some notation that will be used in \rev{subsequent} sections.}
We define $\langle \cdot, \cdot \rangle$ as the standard inner product in $\R^n$:
$\langle x, y \rangle = \sum_{i=1}^{n} x_i y_i,$ for all $x, y \in \R^n.$
We use $\|\cdot \|$ to denote the 2-norm of a vector:
$\|x\| = \langle x, x \rangle ^{1/2}.$
For a real symmetric matrix ${S}$, we \rev{denote} ${S}\succeq 0$ (${S}\succ 0$) if ${S}$ is positive semidefinite (positive definite).

\medskip
\vspace*{1 ex}
\section{ADMM with the BFGS update and its Convergence analysis}\label{section:convergence}
\rev{In this section, we first propose the updating rule of $T_k$ via the BFGS update for VMSP-ADMM, and show a key property on $T_k$ for the convergence.
Then we establish the global convergence.}

\subsection{Construction of the regularized matrix \texorpdfstring{$T_k$}{TEXT} via the BFGS update}
As discussed in Introduction, we propose to construct $T_k$ as $T_k = B_k - M$, where $M = \nabla^2_{xx} \L_{\beta}(x,y,\la)$. We want $T_k$ to be positive semidefinite for global convergence as \rev{a usual semi-proximal ADMM.} Moreover we want $B_k$ to be as close to $M$ as possible for rapid convergence. To this end, we propose to generate $B_k$ by \rev{the BFGS update} with respect to $M$.
Then we may consider the BFGS update with a given $s \in R^n$ and $l=Ms$. Note that  $s^\top l>0$ when $s\neq 0$. Since BFGS usually constructs the inverse of $B_k$, we let
$H_k = B_k^ {-1}.$
Using $H_k$, we can easily solve subproblem \eqref{equ:ADB1}.

Now we  briefly sketch \rev{the} BFGS \rev{update} and \rev{the limited memory BFGS (L-BFGS) \cite{nocedal1980updating,hale2013introduction}.}
Let $s_{k}=x^{k+1}-x^{k}, l_{k}=M s_{k}$.
Then the BFGS \rev{updates} for $B_{k+1}$ and $H_{k+1}$ are given as
\[\label{bfgsb}
B^{BFGS}_{k+1}=B_{k}+{\frac {{l}_{k}{l}_{k}^\top}{{l}_{k}^\top {s}_{k}}}-{\frac {B_{k}{s}_{k}{s}_{k}^\top B_{k}^\top}{{s}_{k}^\top B_{k}{s}_{k}}},
\]

\[\label{bfgsh}
H^{BFGS}_{k+1}=\left(I-{\frac {s_{k}l_{k}^\top}{s_{k}^\top l_{k}}}\right) H_k \left(I-{\frac {l_{k}s_{k}^\top}{s_{k}^\top l_{k}}}\right)+{\frac {s_{k}s_{k}^\top}{s_{k}^\top l_{k}}}.
\]
Since $ s_{k}^\top l_{k} >0$, $B_{k+1}^{BFGS}$ and $H_{k+1}^{BFGS}$ are positive definite whenever $B_k, H_k \succ 0$.
Moreover
$$l_{k}=B_{k+1}^{BFGS} s_{k}~~ \mathrm{and}~~ s_{k} = H_{k+1}^{BFGS} l_{k}.$$

\rev{The BFGS update} requires only matrix-vector multiplications, which brings the computational cost at each iteration to $O(n^2)$. If the number of variables is very large, even $O(n^2)$ per iteration is too expensive \rev{in terms of both CPU time and memory usage.}\par

A less computationally intensive method is the limited memory BFGS method \cite{nocedal1980updating,hale2013introduction}. Instead of updating and storing the entire approximated inverse Hessian matrix, the L-BFGS method uses \rev{the vectors $(s_i, l_i)$ in} the last $h$ iterations and \rev{constructs $H_{k+1}$ by using these vectors.} The updating in L-BFGS brings the computational cost down to $O(hn)$ per iteration.

\medskip
\vspace*{1 ex}
\subsection{Property of the regularized matrix \texorpdfstring{$T_k$}{TEXT} via the BFGS update}
For the global convergence, we need $T_k = B_k - M \succeq 0$, that is $B_k \succeq M$. Note that $B_k \succeq M $ is equivalent to $H_k \preceq M^{-1}$, \rev{where $H_k = (B_k)^{-1}$}.
We will show that $H_k \preceq M^{-1}$ for all $k$ when the initial matrix $H_0$ satisfies
$$H_0 \preceq M^{-1}.$$

We first show  a technical lemma on $s$ and $l$.
\begin{lemma}
\label{lemma:sy}
Let $s \in R^n$ such that $s\neq 0$.
Moreover let $l=M s$ and $\Phi=\{ z\in R^n\;|\; \langle s, z\rangle =0\}$.
Then for any $v\in R^n$, there exist $c\in R$ and $z\in \Phi$ such that $v=cl+z$.
\end{lemma}
\begin{proof}
Let $v\in R^n$.
Then there exist $c_1, c_2 \in R$ and $z^1, z^2\in \Phi$
such that $v=c_1 s+z^1$ and $l=c_2 s +z^2$.  Since $s^\top l>0$, \rev{we have} $c_2 \neq 0$.
Thus $s=\frac{1}{c_2} l-\frac{1}{c_2} z^2$. Substituting it into $v=c_1 s+z^1$ yields
\begin{align*}
v=c_1\left( \frac{1}{c_2} l-\frac{1}{c_2} z^2\right)+z^1=\frac{c_1}{c_2} l+z^1-\frac{c_1}{c_2}z^2.
\end{align*}
Let $c=\frac{c_1}{c_2}$ and $z=z^1-\frac{c_1}{c_2}z^2$. Then $z\in \Phi$ and $v=cl+z$.
\end{proof}

Recall the BFGS \rev{update} (\ref{bfgsh}) is rewritten as
\[\label{eqn:HVS}
\rev{
H_{\rm next} = H - \frac{H l s^\top + s l^\top H} {s^\top l} + \left(1 + \frac{l^\top H l}{s^\top l}\right) \frac{s s^\top}{s^\top l},
}
\]
where $H$ is the proximal matrix for the current step and $H_{\rm next}$ is the new matrix generated via BFGS update.
Moreover we have
\[\label{eqn:HyAy}
H_{\rm next} l =s= M^{-1} l.\]

The following theorem will play a key role \rev{for} the global convergence \rev{of the proposed} method.
\begin{theorem}
\label{theorem:basic}
Let $s \in R^n$ such that $s\neq 0$,  and let $l=M s$.
If $H \preceq M^{-1}$, then $H_{ next} \preceq M^{-1}$.
\end{theorem}
\begin{spacing}{1.3}
\begin{proof}
Let $v$ be an arbitrary nonzero vector in $R^n$.
Let $\Phi=\{ z\in R^n\;|\; \langle s, z\rangle =0\}$.
From Lemma \ref{lemma:sy} there exist $c \in R$ and $z\in \Phi$ such that
$v=cl+z$.
It then follows from (\ref{eqn:HyAy}) and the definition of $z$ that
\begin{eqnarray*}
v^\top H_{\rm next} v &=& (cl+z)^\top H_{\rm next} (cl+z) \\
&=& c^2 l^\top s + 2 cs^\top z + z^\top  H_{\rm next} z \\
&=& c^2 l^\top s + z^\top  H_{\rm next} z  \\
&=& c^2 l^\top  M^{-1} l+ z^\top  H_{\rm next} z.
\end{eqnarray*}
We now consider the last term of the right-hand side of the last equation. Since $z\in \Phi$,
we have
$$
 z^\top \left( \frac{sl^\top}{s^\top l}  H \frac{ls^\top}{s^\top l} \right) z= 0,
$$
$$
 z^\top \left( \frac{sl^\top}{s^\top l}  H \right)  z =0
$$
and
$$
 z^\top \left( \frac{ss^\top}{s^\top l} \right) z = 0.
$$
It then follows from (\ref{eqn:HVS}) that
$$
z^\top H_{\rm next} z = z^\top \rev{H} z - 2 z^\top \left( \frac{s l^\top}{s^\top l} H\right) z + z^\top \left( \frac{s l^\top}{s^\top l} H \frac{l s^\top}{s^\top l}\right) z + \frac{z^\top s s^\top z}{s^\top l}= z^\top H z.
$$
Moreover equation (\ref{eqn:HyAy}) implies
$$
l^\top  M^{-1} z= s^\top z =0.
$$
Consequently we have
\begin{eqnarray*}
v^\top H_{\rm next} v &=&c^2 l^\top   M^{-1} l+ z^\top  H z \\
&\leq &c^2 l^\top  M^{-1} l+ z^\top    M^{-1} z \\
&=& (cl+z)^\top  M^{-1} (cl+z) - 2 cl^\top M^{-1} z \\
&=& v^\top   M^{-1} v,
\end{eqnarray*}
where the inequality follows from the assumption.
Since $v$ is arbitrary, we have $H_{\rm next} \preceq M^{-1}$.
\end{proof}
\end{spacing}
This theorem shows that if $H_0 \preceq M^{-1}$, then $H_k \preceq M^{-1}$, and hence $T_k \succeq 0$.

\subsection{The variable metric semi-proximal ADMM with BFGS and its global convergence} \label{subsection:convergence}
\rev{We propose the following variable metric semi-proximal ADMM with the BFGS update algorithm (ADM-BFGS).}
\IncMargin{1em}
\begin{algorithm}
    \SetKwInOut{Input}{\textbf{Input}}\SetKwInOut{Output}{\textbf{Output}}

    \Input{
        data matrix $A$, initial point $(x^0, y^0, \la^0)$, penalty parameter $\beta$, maxIter;\\
        initial matrix $H_0 \preceq M^{-1}$, constant $\bar{k} \in [1, \infty]$, stopping criterion $\epsilon$.\\}
    \Output{
        \\
        approximative solution $(x^k, y^k, \la^k)$ \\}
    \BlankLine
    initialization\;
     \While{$k< \mathrm{maxIter}$ or not convergence 
     }{

                \eIf{$k \leq \bar{k}$ \rev{$\mathrm{and}$ $x_k - x_{k-1}\neq 0$}}
                {
                        {\bf{update}} $H_{k}$ via BFGS (or L-BFGS) with the initial matrix $H_0$;
                }
                {
                    $H_{k} = H_{k-1}$\;
                }
                {\bf{update}} $x^{k+1}$ by solving the $x-$subproblem:
                $x^{k+1} = x^{k} + H_{k}\left( \lambda^k + \beta y^k + A^\top b - M x^k \right);$

        \nosemic {\bf{update}}  $y^{k+1}$ by solving the $y-$subproblem: \;
        \pushline\dosemic\nonl  $\qquad \;y^{k+1} = \arg\min_{y} \left\{ g(y) - \langle \lambda^k, x^{k+1} - y\rangle + \frac{\beta}{2}\|x^{k+1}-y\|^2+\frac{1}{2}\|y-y^k\|_{S}^2  \right\};$

        \popline {\bf{update}} \rev{Lagrange multipliers:} $\la^{k+1} = \la^k - \beta(x^{k+1} - y^{k+1}).$
        }
   \NoCaptionOfAlgo
   \caption{Variable metric semi-proximal ADMM with the BFGS update (ADM-BFGS)\label{al}}
\end{algorithm}
\DecMargin{1em}
%
%
%
%
\vfill
\newpage
\rev{We now develop a general convergence result for variable metric semi-proximal ADMM \rev{\eqref{equ:ADMMB} for problem \eqref{cp} with a general convex function $f$}. Let $\Omega^*$ be a set of $(x^*, y^*, \la^*)$ satisfying the KKT condition of problem \eqref{cp}.
We assume that $\Omega^*$ is non-empty.}
We give some conditions for sequence $\{T_k\}$ that should be obeyed to guarantee the \rev{global} convergence.
\begin{condition}
\label{cond}
For a sequence $\{T_k\}$ in framework \eqref{equ:ADMMB}, there exist $T \succeq 0$ and a sequence $\{\gamma_k\}$ such that
\begin{description}
{\item[(i)] ~~ $T \preceq T_{k+1} \preceq (1+\gamma_k) T_k$ \ for all \ $k$,
 \item[(ii)] ~~ $\sum\limits_{0}^{\infty} \gamma_k < \infty$ \ and \ $\gamma_k \geq 0$ \ for all \ $k$.
}
\end{description}
\end{condition}
\rev{Under the conditions, we have the following convergence result.}
\begin{theorem}
\label{theo:conv}
Let $\{(x^k, y^k, \lambda^k)\}$ be generated by \eqref{equ:ADMMB}, and let $\{T_k\}$ be a sequence satisfying Condition \ref{cond}. Then sequence $\{(x^k, y^k, \lambda^k)\}$ converges to a point $(x^*, y^*, \la^*) \in \Omega^*$.
\end{theorem}
\begin{proof}
See Appendix \ref{section:convergence property}.
\end{proof}
\rev{
\begin{remark}
Note that Condition \ref{cond} is similar to \cite[Condition C.]{HLHY2002} and \cite[Condition C2]{Goncalves2018}. When sequence $\{T_k\}$ is positive definite, the proof of the global convergence can be found in \cite{HLHY2002}. In the above assumption, the proximal term $T_k$ is assumed to be merely positive semidefinite. Thus we give a proof in Appendix \ref{section:convergence property}.
\end{remark}
}
\rev{Now we give the global convergence of ADM-BFGS as a consequence of Theorem \ref{theo:conv}}.
\begin{theorem}
\label{convergence}
Suppose that sequence $\{T_k\}$ is generated by the BFGS (or L-BFGS) update with $M$. Suppose also that $\{T_k\}$ satisfies Condition \ref{cond}. Then sequence $\{(x^k, y^k, \lambda^k)\}$ generated by ADM-BFGS converges to a point $(x^*, y^*, \la^*) \in \Omega^*$.
\end{theorem}
\begin{proof}
The theorem directly follows from Theorem \ref{theo:conv}.
\end{proof}
\rev{
We currently cannot show that $\{T_k\}$ satisfy Condition \ref{cond} when $\{H_k\}$ is updated by a pure BFGS (or L-BFGS) update and $\bar{k} = \infty$. Hence we give the following two remedies for $T_k$ to be satisfied Condition \ref{cond}.
 \begin{description}
{\item[\color{blue}Remedy 1:]
We let $\bar{k}$ be finite. Then the updating of $B_k$ stopped at $\bar{k}$, that is
\vspace{-1mm}
$$B_{k} = B_{\bar{k}},\;T_{k} = T_{\bar{k}}~~~~ \mathrm{for~~all}~~  k \geq \bar{k},$$
i.e., $\gamma_k = 0$ in Condition \ref{cond} when $ k \geq \bar{k}$. Thus, it is reasonable to say that the sequence $\{T_k\}$ generated by ADM-BFGS and some existing $\{\gamma_k\}$ satisfy the Condition \ref{cond}. Note that the resulting ADM-BFGS becomes ADMM \eqref{equ:sADMM} with $T = T_{\bar{k}}$ for large $k$.

\medskip
 \item[\color{blue}Remedy 2:]  Suppose that $\bar{k} = \infty$. We generate $\{B_k\}$ as follows:
 \vspace{-1mm}
\[\label{cond2}
B_{k+1}=B_{k}+ c_k \left(\frac {{\tilde{l}}_{k}{\tilde{l}}_{k}^\top}{{\tilde{l}}_{k}^\top {s}_{k}} -{\frac {B_{k}{s}_{k}{s}_{k}^\top B_{k}^\top}{{s}_{k}^\top B_{k}{s}_{k}}} \right),
\vspace{-1mm}\]
where $\tilde{l}_k = M s_k + \delta s_k$ with $\delta > 0$, and $\{c_k\}$ is a sequence such that $c_k \in [0,1],$ and $\sum\limits_{k=0}^{\infty} c_k < \infty$.

Now we show that Condition \ref{cond} holds when $B_0 \succeq M + \delta I$. Suppose that $B_0 \succeq M + \delta I$.
Note that $B_{k+1} = B_k + c_k( \bar{B}_{k+1} - B_k)$, where $\bar{B}_{k+1} $ is updated by the pure BFGS update \eqref{bfgsb} with $s_k$ and $\tilde{l}_k$. From Theorem \ref{theorem:basic}, $\bar{B}_{k+1}  \succeq M +\delta I$ when $B_k \succeq M + \delta I$. Since $B_{k+1} = c_k \bar{B}_{k+1} + (1-c_k) B_k$, we have $B_{k+1} \succeq  M + \delta I$, and hence $T_{k+1} = B_{k+1} - M \succeq \delta I \succ 0$.
Therefore the first matrix inequality in Condition \ref{cond} (i) holds.

Next we show the second inequality in Condition \ref{cond} (i) holds. Note that ${s}_{k}^\top B_{k}{s}_{k} \geq \delta \|s_k\|^2$, $\tilde{l}_k^\top s_k = s_k^\top M s_k + \delta \|s_k\|^2 \geq \delta \|s_k\|^2$, and $M$ is the constant matrix. Therefore, $\|\bar{B}_{k+1} - B_k\|$ is bounded above by some $Q>0$, that is, $\|\bar{B}_{k+1} - B_k\| \leq Q$. Then we have
\vspace{-1mm}
$$
c_k( \bar{B}_{k+1} - B_k) \preceq c_k \|\bar{B}_{k+1} - B_k\| \cdot I \preceq c_k \frac Q \delta \cdot \delta I
\preceq c_k \frac Q \delta T_k.
$$
Therefore,
\vspace{-2mm}
\begin{align*}
T_{k+1} ={}& B_{k+1} - M \nn \\
={} & B_k + c_k( \bar{B}_{k+1} - B_k) - M\nn \\
={} & T_k + c_k( \bar{B}_{k+1} - B_k)  \nn \\
\preceq{} & T_k + \frac {c_k Q} {\delta} T_k \nn \\
={} & (1 + \frac {c_k Q} {\delta}) T_k.
\end{align*}
Let $\gamma_k = \frac Q \delta c_k$. Then $T_{k+1} \preceq (1+\gamma_k) T_k$.

Finally we show that Condition \ref{cond} (ii) holds. From the definition of $\gamma_k$, we have $\sum\limits_{k=0}^{\infty} \gamma_k = \frac Q \delta \sum\limits_{k=0}^{\infty} c_k < \infty$.
}
\end{description}
We will present numerical results for BFGS with Remedy 2 and L-BFGS with Remedy 1 in the next section.
}


\section{Numerical results}\label{section:numerical}

In this section, we demonstrate the potential efficiency of our method by some numerical experiments.
\rev{All the experiments are implemented by Matlab R2018b on Windows 10 pro with a 2.10 GHz Intel Xeon E5-2620 v4 processor and 128 GB of RAM.}

\subsection{Detail settings in the numerical experiments}\label{num:set}
\rev{In this subsection, we give the detail settings in the numerical experiments.}

     \subsubsection{Test problems}  \label{num:subsub:pro}
 We consider \rev{to solve} the Lasso problem:
 \vspace{-1mm}
\[\label{prob:lasso}
 \min_{x \in \R^n} \ \frac{1}{2}\|Ax-b\|_2^{2}+ \tau\|x\|_1,
 \]
where
 \begin{itemize}
{\item~~  $A \in \R^{m \times n}$ is \rev{a} given data matrix;
 \item~~  $x \in \R^n$ is \rev{a} vector of feature coefficients to be estimated;
 \item~~  $b \in \R^m$ is \rev{an} observation vector and $\tau \in \R$ is \rev{a positive} regularization parameter;
 \item~~  $m$ is the number of data points, and $n$ is the number of features.
}
\end{itemize}
By introducing an auxiliary variable $y\in \R^{n}$, we reformulate problem \eqref{prob:lasso} as
\[\label{prob:lasso2}
\min_{ x\in \R^{n},\  y\in \R^{n}} \  \ \frac{1}{2}\|Ax- b\|_2^{2}+ \tau \|y\|_1 \quad \st \quad  x-y=0.
\]

\rev{We randomly generate $A$ and $b$ as follows. We first randomly choose $\bar{x}\in \R^n$ with the sparsity $s$, i.e., the number of nonzero elements in $\bar{x}$ over $n$ is $s$. The Matlab code is given as}
\begin{verbatim}
       xbar = sprandn(n,1,s).
\end{verbatim}
\rev{We generate $A$ by the standard normal $\N(0,1)$ distribution whose sparsity density is $p$:}
\begin{verbatim}
       A = sprandn(m,n,p). % N(0,1) with the density p
\end{verbatim}
\rev{Then we calculate $b = A \bar{x} + \varrho$, where $\varrho$ is a noise under $\N(0,10^{-3})$ distribution. The Matlab code is}
\begin{verbatim}
       b = A*xbar + sqrt(0.001)*randn(m,1).
\end{verbatim}
\rev{The regularization parameter is set to $\tau = 0.1 \tau_{max}$, where $\tau_{max} = \|A^\top b\|_{\infty}$:}
\begin{verbatim}
        tau = 0.1* norm(A'*b, 'inf').
\end{verbatim}

 \subsubsection{Test ADMMs} \label{num:subsub:admm}
 \rev{In the numerical experiments, we test the following 7 ADMMs. The differences of the ADMMs are choices of the proximal term $T_k$ in $x$-subproblems.
        \begin{description}
    {

        \item[{\color{blue}ADM-OPT:}] the classical ADMM \cite{gabay1976dual,glowinski1975approximation}. ADM-OPT solves the original subproblem \eqref{alx} exactly, that is, $T_k=0$ for all $k$;

       \item[{\color{blue}ADM-SPRO:}] the semi-proximal ADMM in \cite{fazel2013hankel}. A positive semidefinite matrix $T_k$ is chosen as
\[\label{proxt}
T_k = \xi I  - \beta I  - A^{\Tsf}A \ \mbox{with}\
\xi = \kappa_1 * \lambda_{\max} \left( \beta I +  A^{\Tsf}A\right),\;\; \kappa_1 >1, \;\; \mathrm{for\ all} \;\; k;
\]
  \item[{\color{blue}ADM-IPRO:}] the indefinite proximal ADMM based on \rev{\cite{li2016majorized, he2017optimal}}. An indefinite proximal matrix $T_k$ is chosen as
\[\label{proxi}
T_k = \xi I - A^{\Tsf}A \ \mbox{with}\
\xi = \kappa_2 * \lambda_{\max} \left( A^{\Tsf}A\right), \;\; \kappa_2 > 0.75,  \;\; \mathrm{for\ all} \;\; k;
\]
       \item[{\color{blue}ADM-BFGS:}] the proximal ADMM with the BFGS update with $\bar{k}=\infty$. An initial matrix of $B_k$ (or $H_k$) is given as
\[\label{proxb}
{\color{blue}B_0 = \rev{\xi} I, \;\; \rev{\xi} = \rev{\kappa_3}*\lambda_{\max}(\beta I + A^\top A)}, \;\; \kappa_3 > 0.75;
\]
       \item[{\color{blue}ADM-LBFGS:}] the proximal ADMM with the L-BFGS update with $\bar{k}=\infty$. The initial semidefinite proximal matrix for limited memory BFGS is the same as \eqref{proxb}. Note that $H_0 = \frac{1}{\rev{\xi}} I$. We fix $H_0^k = H_0$ of each updating step for L-BFGS matrix;
       \item[{\color{blue}ADM-BFGS-R:}] the proximal ADMM with Remedy 2 given in Subsection \ref{section:convergence};
       \item[{\color{blue}ADM-LBFGS-R:}] the proximal ADMM with Remedy 1, that is ADM-LBFGS with $\bar{k} < \infty$.
    }
    \end{description}
 }
\rev{These ADMMs except for ADM-OPT need the maximum eigenvalues $\lambda_{\max} \left( \beta I +  A^\top A\right)$ and $\lambda_{\max} \left(A^\top A\right)$.
        We adopt the following Matlab codes to compute these eigenvalues:}
\begin{verbatim}
            eig_max = svds(A,1)^2 + beta;
            eig_max = svds(A,1)^2.
\end{verbatim}
\rev{While ADM-OPT must solve the unconstrained quadratic optimization \eqref{alx}.
We use a Cholesky factorization for solving it.
When $m= n$, we use ``chol'' in Matlab for $(A^\top A + \beta I)$. When $A$ is fat (i.e., $m<n$), we apply the Sherman-Morrison formula to $(\beta I +  A^\top A)^{-1}$ as
$$
(\beta I +  A^\top A)^{-1} = \frac 1\beta I - \frac{1}{\beta^2} \cdot A^\top \cdot\left( I + \frac 1\beta A A^\top \right)^{-1} \cdot A,
$$
and compute the factorization $L L^\top$ of a smaller matrix $(I + (1/\beta)A A^\top)$ by the ``chol'' function.
Then the $x$-subproblems are solved as}
\begin{verbatim}
            q = A'*b + lambda + beta*y;
            x = q/beta - (A'*(L' \ ( L \ (A*q) )))/beta^2.
\end{verbatim}
\rev{
Note that the Cholesky factorization of $(A^\top A + \beta I)$ or $(I + (1/\beta)A A^\top)$ is calculated only once for each test problem.
}

\subsubsection{Other setting and notations} \label{num:subsub:set}
\rev{
 \noindent \textbf{Stopping criterion:}~~
 We adopt the same stopping criterion as in \cite{boyd2011distributed} for all the numerical experiments, \rev{that is, if the primal and dual residuals $r^k$ and $\sigma^k$ satisfy}
\[\label{stop}
\|r^k\|_2 \leq \epsilon_k^{\mathrm{pri}} ~~\mathrm{and}~~ \|\sigma^k\|_2 \leq \epsilon_k ^{\mathrm{dual}},
\]
\rev{then we stop the algorithms, where $r^k = x^k - y^k$, $\sigma^k = -\beta(y^k - y^{k-1})$, and} $\epsilon^{\mathrm{pri}} > 0$ and $\epsilon ^{\mathrm{dual}} > 0$ are feasibility tolerances for the primal and dual feasibility conditions, respectively. These tolerances can be chosen using an absolute and relative criterion from the suggestion in \cite{boyd2011distributed}, such as
\begin{eqnarray*}
\epsilon_k^{\mathrm{pri}} &=& \sqrt{n} \epsilon ^ \mathrm{abs} + \epsilon^ \mathrm{rel} \mathrm{max}\{\|x^k\|_2, \|-y^k\|_2 \}, \\
\epsilon_k ^{\mathrm{dual}} &=& \sqrt{n} \epsilon ^ \mathrm{abs} + \epsilon^ \mathrm{rel} \|\lambda ^k\|_2,
\end{eqnarray*}
where $\epsilon ^ \mathrm{abs}>0$ is an absolute tolerance and $\epsilon^ \mathrm{rel} >0$ is a relative tolerance.

The stopping criterions are set to $\epsilon ^ \mathrm{abs} = 10^{-4}$ and $\epsilon^ \mathrm{rel} = 10^{-3}$ in all experiments.

\medskip
\noindent \textbf{Other setting:}~~We always choose $S = 0$ in \eqref{equ:ADB2}. We set the initial points as $x^0 = y^0 = 0$ and $\lambda^0 = 0$. The maximum iterations are set to be $20000$ in all experiments.

\medskip
\noindent \textbf{Notations in tables for numerical results:}
     \begin{itemize}
     {   \item[$\bullet$] Iter.: the iteration steps for each algorithm;
         \item[$\bullet$] Time: the total CPU time for each algorithm;
         \item[$\bullet$] T-L: the CPU time for the Cholesky factorization and the calculation of $A A^\top$ or $A^\top A$;
         \item[$\bullet$] T-ME: the CPU time of computing for the maximum eigenvalue;
         \item[$\bullet$] T-A: the CPU time for the algorithm proceed without T-L or T-ME;
         \item[$\bullet$] T-QN: the CPU time for BFGS update (matrix $H_k$) of ADM-BFGS.
    }
    \end{itemize}
    All of the CPU time is recorded in seconds.

}
\subsection{Test I: ADMM with the BFGS update}\label{num:sub1}
\rev{In the subsection, we first compare four different methods: ADM-OPT, ADM-SPRO with $\kappa_1 = 1.01$, ADM-IPRO with $\kappa_2 = 0.8$, and ADM-BFGS with $\kappa_3 = 1.01$. We also present numerical results for ADMM with Remedy 2 given in Subsection \ref{subsection:convergence} for the global convergence.

We solve problem \eqref{prob:lasso2} with $n = 2000$, $m = 1000$, $s=0.1$ and $p \in \{0.1, 0.5\}$. All of the other settings and calculations are followed from Subsection \ref{num:set}.
We solve 10 problems in each test, and Table \ref{test1-result1} shows the average of iterative steps and CPU time.
\begin{table}[!htp]
\vspace{-1mm}
\centering\caption{Comparison on iteration steps and CPU time (seconds) among the methods}\par
\begin{tabular}{c c c c| c| c c |c c| c c| c c c} \hline
\multicolumn{4}{c|}{Problem} & \multirow{2}{*}{$\beta$} & \multicolumn{2}{c|}{ADM-OPT} &\multicolumn{2}{c|}{ADM-SPRO}  &\multicolumn{2}{c|}{ADM-IPRO} &\multicolumn{3}{c}{ADM-BFGS} \cr \cline{1-4}\cline{6-14}
$n$ & $m$ & $s$ & $p$ & ~ & Iter. & Time & Iter. &Time & Iter. &Time  &Iter. &Time & T-QN \cr \hline
2000 & 1000& 0.1 & $0.1$ & $100$ & 20.5 & 0.21 & 64.3 & 0.15 & 54.3 & 0.14 & 38.4 & 3.64 & 2.89 \\
2000 & 1000& 0.1 & $0.5$ & $100$ & 63.1 & 0.67 & 197.9 & 0.45 & 160.0 & 0.41 & 71.4 & 7.35 & 5.49 \\
2000 & 1000& 0.1 & $0.5$ & $500$ & 20.9 & 0.55 & 68.5 & 0.32 & 58.4 & 0.31 & 37.3 & 4.50 & 2.84 \\
\hline
\end{tabular}\label{test1-result1}
\end{table}

From the table, it is obvious to see that the classical ADMM find solutions within least iterative steps, while the indefinite proximal ADMM admits the faster one at the CPU time. The ADMM with BFGS can get solutions with relatively less iterations. However it spends much time to compute the $H_k$ as indicated in the column of T-QN. When data matrix $A$ is ill condition or it is impossible to compute the inverse of Hessian matrix of augmented Lagrangian function, it is meaningful to use the matrix $H_k$ since it can get a solution with less iterative steps.
}

\rev{
Next, we give numerical results on iteration steps of ADM-BFGS-R in Table \ref{c-result}. We update $c_k$ by $c_k = \zeta^k$ with $\zeta\in [0,1]$, and chose a positive $\delta \in \{$100, 1e-5$\}$.
We solve problem \eqref{prob:lasso2} with the same settings as those in Table \ref{test1-result1}, that is, $n = 2000$, $m = 1000$, $s=0.1$ and $p \in \{0.1, 0.5\}$.
The results are compared with ADM-BFGS.

\begin{table}[!htp]
\centering\caption{Results for Remedy 2 with different $\delta$ and $c_k\; (c_k = \zeta^k)$}
 \resizebox{\linewidth}{!}{
\begin{tabular}{ c c c c| c| c | c c c | c c c } \hline
\multicolumn{4}{c|}{\multirow{2}{*}{Problem}}  & \multirow{3}{*}{$\beta$} &\multirow{2}{*}{ADM-BFGS} &\multicolumn{3}{c|}{ADM-BFGS-R $\;\;\delta$ = 100} &\multicolumn{3}{c}{ADM-BFGS-R $\;\;\delta$ = 1e-5}  \cr \cline{7-12}
~ & ~ & ~ & ~ & ~ & ~ &{$\zeta=0.1$} &{$\zeta=0.5$} &{$\zeta=0.99$} &{$\zeta=0.1$} &{$\zeta=0.5$} &{$\zeta=0.99$}  \cr \cline{1-4}\cline{6-12}
$n$ & $m$ & $s$ & $p$ & ~ & Iter. & Iter. & Iter. & Iter. & Iter. & Iter. & Iter. \cr \hline
2000 & 1000& 0.1 & $0.1$ & $100$ & 38.4 & 75.7 & 71.3 & 43.9 & 67.6 & 63.0 & 38.8 \\
2000 & 1000& 0.1 & $0.5$ & $100$ & 71.4 & 204.4 & 198.6 & 74.7 & 197.0 & 190.4 & 71.8 \\
2000 & 1000& 0.1 & $0.5$ & $500$ & 37.3 & 74.0 & 68.7 & 39.4 & 71.9 & 66.6 & 37.8 \\
 \hline
\end{tabular}}\label{c-result}
\end{table}
Table \ref{c-result} shows that for each $\delta>0$ and $\zeta\in [0,1]$, ADM-BFGS-R can find a solution. When $\delta$ is close to 0 and $\zeta$ is close to 1, the iterative steps of ADM-BFGS-R approach those of ADM-BFGS .
}

\subsection{Test II: ADMM with limited memory BFGS update}\label{num:sub2}

\rev{In this subsection, we test how the ADMM with limited memory BFGS (ADM-LBFGS) works.

We set the number {\color{blue} $h$} of vectors stored in L-BFGS to 10. The comparisons are among the ADM-OPT, ADM-IPRO, and the proposed ADM-LBFGS. We consider large scale problems with $n=10000$ and $s=0.1$.
}
\subsubsection{Behaviors of ADMMs for different \texorpdfstring{$\beta$}{TEXT}}\label{num:subsub1}
\rev{
We first see behavious of ADMMs for different $\beta$. We solve problem \eqref{prob:lasso2} with $m\in \{1000, 5000, 10000\}$ and $p \in \{0.1, 0.5, 1\}$.
We take $\kappa_2 = 0.8$ for ADM-IPRO \eqref{proxi} and $\kappa_{3} = 1.01$ for ADM-LBFGS \eqref{proxb}.
The other settings of the test problems are given in Subsection \ref{num:subsub:set}.

The results of iteration steps and CPU time (seconds) averaged over 10 random trials are shown in Table \ref{test-beta2}.}
\begin{table}[!ht]
\centering\caption{Comparison among ADMMs for different $\beta$}
 \resizebox{\linewidth}{!}{
\begin{tabular}{c| c| c c c c| c c c c| c c c c} \hline
\multirow{2}{*}{Size} & \multirow{2}{*}{$\beta$} &\multicolumn{4}{c|}{ADM-OPT}& \multicolumn{4}{c|}{ADM-IPRO}  &\multicolumn{4}{c}{ADM-LBFGS} \cr\cline{3-14}
~ & ~ & Iter. & Time & T-A & T-L & Iter. & Time &T-A & T-ME & Iter. & Time &T-A & T-ME  \cr \hline
$m$=1000 &50 & 90.9 & 0.66 & 0.34 & 0.32 & 280.1 & 0.92 & 0.46 & 0.46 & 247.3 & 1.27 & 0.81 & 0.46  \\
$p$=0.1  &100 & 59.1 & 0.56 & 0.22 & 0.34 & 175.7 & 0.74 & 0.29 & 0.45 & 182.2 & 1.04 & 0.59 & 0.45  \\
~& 150 & 67.7 & 0.57 & 0.25 & 0.32 & 197.9 & 0.77 & 0.33 & 0.44 & 157.3 & 0.95 & 0.51 & 0.44  \\
~& 200 & 88.4 & 0.67 & 0.33 & 0.34 & 240.2 & 0.82 & 0.38 & 0.44 & 147.4 & 0.90 & 0.46 & 0.44  \\
[5pt]  \hline
$m$=1000 & 200 & 93.3 & 4.63 & 0.99 & 3.64 & 288.6 & 4.40 & 2.46 & 1.94 & 247.5 & 4.57 & 2.63 & 1.94 \\
$p$=0.5  & 300 & 69.9 & 4.38 & 0.73 & 3.65 & 218.8 & 3.79 & 1.83 & 1.96 & 209.0 & 4.15 & 2.19 & 1.96 \\
~ & 500 & 60.3 & 4.37 & 0.62 & 3.75 & 179.2 & 3.39 & 1.49 & 1.90 & 174.2 & 3.72 & 1.82 & 1.90 \\
~ & 800 & 85.0 & 4.29 & 0.87 & 3.42 & 240.6 & 3.91 & 1.98 & 1.93 & 143.4 & 3.42 & 1.49 & 1.93 \\
[5pt]  \hline
$m$=1000 & 200 & 131.7 & 9.73 & 2.01 & 7.72 & 428.2 & 9.02 & 5.84 & 3.18 & 325.9 & 8.29 & 5.11 & 3.18 \\
$p$=1  & 300 & 97.3 & 9.13 & 1.53 & 7.60 & 305.4 & 7.43 & 4.24 & 3.19 & 273.5 & 7.56 & 4.37 & 3.19 \\
~ & 500 & 67.8 & 8.97 & 1.05 & 7.92 & 213.0 & 6.05 & 2.85 & 3.20 & 220.9 & 6.66 & 3.46 & 3.20 \\
~ & 800 & 67.5 & 9.17 & 1.05 & 8.12 & 195.9 & 5.85 & 2.65 & 3.20 & 183.9 & 6.08 & 2.88 & 3.20 \\
[5pt] \hline
$m$=5000 & 100 & 80.4 & 15.67 & 4.48 & 11.19 & 176.6 & 7.24 & 2.29 & 4.95 & 86.6 & 6.25 & 1.30 & 4.95 \\
$p$=0.1 & 200 & 41.2 & 13.90 & 2.34 & 11.56 & 103.1 & 6.24 & 1.32 & 4.92 & 55.9 & 5.75 & 0.83 & 4.92 \\
~ & 500 & 20.4 & 12.38 & 1.14 & 11.24 & 51.2 & 5.62 & 0.64 & 4.98 & 38.0 & 5.55 & 0.57 & 4.98 \\
~ & 800 & 22.0 & 12.87 & 1.27 & 11.60 & 61.0 & 5.81 & 0.78 & 5.03 & 37.2 & 5.58 & 0.55 & 5.03 \\
[5pt]  \hline
$m$=5000 & 500 & 67.1 & 95.46 & 6.36 & 89.10 & 151.3 & 24.45 & 7.54 & 16.91 & 75.6 & 20.84 & 3.93 & 16.91 \\
$p$=0.5 & 1000 & 34.3 & 92.83 & 3.02 & 89.81 & 86.7 & 21.00 & 4.09 & 16.91 & 50.2 & 19.44 & 2.53 & 16.91 \\
~ & 2000 & 20.4 & 93.28 & 1.79 & 91.49 & 51.4 & 19.91 & 2.36 & 17.55 & 38.1 & 19.45 & 1.90 & 17.55 \\
~ & 2500 & 20.6 & 91.92 & 1.88 & 90.04 & 53.6 & 19.94 & 2.61 & 17.33 & 36.0& 19.13 & 1.80 & 17.33 \\
[5pt]  \hline
$m$=5000 & 1000 & 52.9 & 201.69 & 6.16 & 195.53 & 130.8 & 33.67 & 10.00 & 23.67 & 65.6 & 28.82 & 5.15 & 23.67 \\
$p$=1 & 2000 & 27.2 & 200.39 & 3.07 & 197.32 & 73.0 & 29.23 & 5.39 & 23.84 & 44.9 & 27.27 & 3.43 & 23.84 \\
~ & 3000 & 20.3 & 210.74 & 2.45 & 208.29 & 55.0 & 31.01 & 4.16 & 26.85 & 39.3 & 29.98 & 3.13 & 26.85 \\
~ & 3200 & 20.7 & 208.26 & 2.54 & 205.72 & 53.2 & 30.78 & 4.08 & 26.70 & 38.8 & 29.81 & 3.11 & 26.70 \\
[5pt]  \hline
$m$=10000 & 200 & 59.2 & 59.57 & 10.72 & 48.85 & 100.3 & 15.90 & 3.02 & 12.88 & 60.2 & 14.85 & 1.97 & 12.88 \\
$p$=0.1 & 500 & 24.5 & 51.97 & 4.75 & 47.22 & 48.0 & 14.39 & 1.50 & 12.89 & 30.6 & 13.93 & 1.04 & 12.89 \\
~ & 1000 & 15.9 & 51.13 & 3.06 & 48.07 & 30.0 & 13.82 & 0.95 & 12.87 & 23.7 & 13.66 & 0.79 & 12.87 \\
~ & 1500 & 16.9 & 50.61 & 3.19 & 47.42 & 34.4 & 13.88 & 1.08 & 12.80 & 24.4 & 13.62 & 0.82 & 12.80 \\
[5pt]  \hline
$m$=10000 & 1000 & 49.8 & 426.04 & 9.69 & 416.35 & 88.2 & 52.69 & 9.42 & 43.27 & 50.2 & 48.77 & 5.50 & 43.27 \\
$p$=0.5 & 2000 & 25.7 & 413.19 & 5.20 & 407.99 & 49.0 & 48.17 & 5.40 & 42.77 & 31.2 & 46.26 & 3.49 & 42.77 \\
~ & 3000 & 17.6 & 432.36 & 3.28 & 429.08 & 38.6 & 47.64 & 4.26 & 43.38 & 26.0 & 46.29 & 2.91 & 43.38 \\
~ & 3500 & 15.9 & 408.71 & 2.86 & 405.85 & 34.4 & 45.14 & 3.31 & 41.83 & 25.0 & 44.37 & 2.54 & 41.83 \\
[5pt]  \hline
$m$=10000 & 2000 & 40.8 & 983.12 & 6.91 & 976.21 & 72.0 & 67.15 & 10.49 & 56.66 & 42.6 & 62.94 & 6.28 & 56.66 \\
$p$=1 & 3000 & 27.6 & 948.20 & 4.68 & 943.52 & 52.2 & 64.34 & 7.60 & 56.74 & 34.0 & 61.74 & 5.00 & 56.74 \\
~ & 5000 & 17.4 & 947.22 & 3.06 & 944.16 & 36.4 & 62.45 & 5.34 & 57.11 & 26.0 & 60.97 & 3.86 & 57.11 \\
~ & 5500 & 16.4 & 931.24 & 2.87 & 928.37 & 33.8 & 61.88 & 4.96 & 56.92 & 25.0 & 60.62 & 3.70 & 56.92 \\
\hline
\end{tabular}}\label{test-beta2}
\end{table}

\medskip
\rev{From Table \ref{test-beta2}, we can observe that the ADMM with L-BFGS performs well for different $\beta$.
In each case, ADM-LBFGS can find solutions within the same level CPU time for the algorithm proceed (T-A) as the classical ADMM (ADM-OPT).
ADM-OPT appears to be the best method to find a solution within least iterations and CPU time when the size $m=1000$ and sparsity $p=0.1$. However, it becomes slower due to the CPU time for Cholesky factorization (T-L) when $p=0.5$ and 1.
Note that the T-L takes much time as compared to the computations of the maximum eigenvalue (T-ME) for ADM-LBFGS and ADM-IPRO when the size $m$ of matrix $A$ is larger than 1000, especially when $A$ is less sparse with $p=0.5$ and 1.
Comparing with ADM-OPT when $m=5000$, ADM-LBFGS can reduce the CPU time at about 50$\%$ for the sparse case $p=0.1$, and about 80$\%$ for the hard cases where $p=0.5$ and 1.
Besides, for a large and dense matrix $A$ with $m=10000$ and $p=1$, ADM-LBFGS can reduce the CPU time by 93$\%$ as compared to ADM-OPT.
On the other hand, ADM-LBFGS is a little faster than ADM-IPRO as $\xi$ is sophisticatedly chosen with the maximum eigenvalue.
}

\subsubsection{Behaviors of ADM-PRO and ADM-LBFGS for some different \texorpdfstring{$\xi$}{TEXT}}\label{num:subsub2}
\rev{
In the above experiments, we have chosen $\xi = 0.8* \lambda_{\max} \left( A^{\Tsf}A\right)$ for the indefinite proximal term and $\xi = 1.01*\lambda_{\max}(\beta I + A^\top A)$ for the semidefinite proximal term. This is unrealistic for some large scale applications where the calculation of maximum eigenvalue is expensive.
Next we test the behaviours of ADM-LBFGS and proximal ADMM (ADM-IPRO) with different $\kappa_2$ and $\kappa_3$.

We solve problem \eqref{prob:lasso2} with $m=5000$ and $p\in \{0.5,1\}$. Since the results in Table \ref{test-beta2} for $m=5000$ indicate that a reasonable $\beta$ is around 2000, we take $\beta \in \{1000,2000,3000\}$ in this experiments. We also take $\kappa_2, \kappa_3 \in \{0.75, 0.8, 1.01, 5.0, 10.0, 100\}$ in \eqref{proxi} and \eqref{proxb}.
%
Other settings and notations are given in Subsection \ref{num:subsub:set}.
Table \ref{kappa-result} shows the results of iteration steps and CPU time (seconds) averaged over 10 random trials for every $\kappa_2$ and $\kappa_3$.
\begin{table}[!ht]
\vspace{-1mm}
\centering\caption{Different $\kappa_2$ and $\kappa_3$ for proximal ADMM}
\begin{tabular}{c| c| c c c| c c c} \hline
\multirow{2}{*}{Setting} & \multirow{2}{*}{$\kappa_2, \kappa_3$}  &\multicolumn{3}{c|}{ADM-IPRO}  &\multicolumn{3}{c}{ADM-LBFGS}  \cr\cline{3-8}
~ & ~ & Iter. &T-A & Time & Iter. &T-A & Time\cr \hline
$p=0.5$ & 0.75 & 74.6 & 3.40 & 20.20 & 46.4 & 2.21 & 19.01 \\
$\beta = 1000$ & 0.80 & 80.3 & 3.67 & 20.47 & 47.3 & 2.26 & 19.06  \\
T-ME = 16.80s & 1.01 & 99.6 & 4.56 & 21.36 & 49.7 & 2.42 & 19.22  \\
~ & 5.0 & 400.6 & 18.44 & 35.24 & 102.9 & 5.01 & 21.81  \\
~ & 10.0 & 734.2 & 33.57 & 50.37 & 147.7 & 7.02 & 23.82   \\
~ & 100.0 & 4626.3 & 211.53 & 228.33 & 330.7 & 15.91 & 32.71\\
[6pt]\hline
$p=0.5$ & 0.75 & 42.0 & 1.95 & 18.30 & 33.6 & 1.62 & 17.97  \\
$\beta = 2000$ & 0.80 & 44.3 & 2.06 & 18.41 & 34.7 & 1.68 & 18.03   \\
T-ME = 16.35s & 1.01 & 53.8 & 2.51 & 18.86 & 38.2 & 1.85 & 18.20   \\
~ & 5.0 & 213.5 & 9.85 & 26.20 & 100.4 & 4.89 & 21.24  \\
~ & 10.0 & 376.6 & 17.39 & 33.74 & 137.2 & 6.63 & 22.98 \\
~ & 100.0 & 2324.4 & 107.59 & 123.94 & 327.7 & 16.02 & 32.37  \\
[6pt] \hline
$p=1$ & 0.75 & 62.0 & 4.38 & 30.38 & 41.0 & 3.00 & 29.00  \\
$\beta = 2000$ & 0.80 & 64.9 & 4.58 & 30.58 & 41.5 & 3.02 & 29.02  \\
T-ME = 26.00s  & 1.01 & 81.8 & 5.78 & 31.78  & 44.6 & 3.22 & 29.22 \\
~ & 5.0 & 334.7 & 23.59 & 49.59 & 105.3 & 7.71 & 33.71  \\
~ & 10.0 & 593.8 & 41.76 & 67.76 & 145.1 & 10.67 & 36.67 \\
~ & 100.0 & 3623.6 & 254.90 & 280.90 & 344.2 & 25.06 & 51.06 \\
[6pt] \hline
$p=1$ & 0.75 & 43.8 & 3.13 & 27.29 & 34.5 & 2.49 & 26.65  \\
$\beta = 3000$ & 0.80 & 46.0 & 3.28 & 27.44 & 35.5 & 2.59 & 26.75    \\
T-ME = 24.16s & 1.01 & 56.5 & 4.02 & 28.18 & 39.1 & 2.85 & 27.01 \\
~ & 5.0 & 222.6 & 15.61 & 39.77 & 104.7 & 7.58 & 31.74  \\
~ & 10.0 & 405.5 & 28.40 & 52.56 & 139.8 & 10.11 & 34.27  \\
~ & 100.0 & 2467.4 & 172.53 & 196.69 & 359.7 & 26.17 & 50.33  \\ \hline

\end{tabular}\label{kappa-result}
\end{table}

From Table \ref{kappa-result}, we see that ADM-LBFGS always works well and remains stable. Note that ADM-LBFGS is a little faster than ADM-IPRO when $\xi$ is chosen nearly around the maximum eigenvalue, $\kappa_2, \kappa_3=0.75, 0.80, 1.01$ for instance. On average, it can lead to a $30\%$ reduction in the number of iterations. There are no much differences in the CPU time because T-ME counts for a lot. When $\kappa_2, \kappa_3 = 100$ which are chosen far away from maximum eigenvalue, ADM-LBFGS can always bring out 85-90$\%$ improvement in the number of iterations
and 75-85$\%$ improvement in the CPU time
as compared to ADM-IPRO. Moreover, we find that ADM-LBFGS also works well even when the proximal term is a slight indefinite matrix, i.e., $\kappa_3 < 1$.
}


\subsubsection{Remedy 1: ADM-LBFGS stops updating of \texorpdfstring{$H_k$}{TEXT} for some finite \texorpdfstring{$\bar{k}$}{TEXT}}\label{num:sub3}
\rev{Finally, we investigate the behavior of ADM-LBFGS-R with various $\bar{k}$ when the updating of $H_k$ stops.

We solve problem \eqref{prob:lasso2} with $m=5000$, $p\in \{0.5,1\}$, $\beta\in\{1000,2000\}$, and set $\bar{k} = \{5,10,20,40,50,100\}$ and $\kappa_3 = 1.01$ in \eqref{proxb}. All the other settings are same as the above experiments.
The results of CPU time and iterations of different stopping $\bar{k}$ averaged over 10 random trials are provided in Table \ref{k-result}.

\begin{table}[!htp]
\centering\caption{Results for stopping at different $\bar{k}$}
\begin{tabular}{c |c |c c c} \hline
\multirow{2}{*}{~} &\multirow{2}{*}{$\bar{k}$} &\multicolumn{3}{c}{ADM-LBFGS-R}  \cr\cline{3-5}
~ & ~ & Iter. & T-A & Time \cr \hline
$p=0.5$ & 5 &  106.0 & 4.89 & 19.25 \\
$\beta = 1000$ & 10 & 93.2 & 4.35 & 18.51 \\
T-ME = 14.16s & 20 & 76.8 & 3.63 & 17.79  \\
~ & 40 & 54.2 & 2.56 & 16.72 \\
~ & 50 & 50.2 & 2.43 & 16.59 \\
~ & 100 & 50.1 & 2.42 & 16.58 \\
[6pt] \hline
$p=1$ & 5 & 85.7 & 6.06 & 32.76 \\
$\beta = 2000$ & 10 & 76.1 & 5.37 & 32.07  \\
T-ME = 26.70s &20 & 62.9 & 4.49 & 31.19 \\
~ & 40 & 43.8 & 3.15 & 29.85 \\
~ & 50 & 44.3 & 3.19 & 29.89 \\
~ & 100 & 44.3 & 3.22 & 29.92 \\
\hline
\end{tabular}\label{k-result}
\end{table}

From the above results, we can see that for all $\bar{k}$, the ADM-LBFGS-R can find a solution within the maximum iteration. In particular, the results for $\bar{k} = 50$ and 100 are almost same, which indicates that $T_{50}$ is a well-tuned proximal matrix for the test problems.
}

\subsection{Conclusions of the numerical experiments \label{sub:con}}
\rev{
From all the above numerical results we conclude that
\begin{itemize}
\item[1.] As compared with the classical ADMM (ADM-OPT), ADM-LBFGS is suitable for dense large scale problems because the calculation of the inverse of $A^\top A$ is not necessary for ADM-LBFGS;
\item[2.] ADM-LBFGS always outperforms the general proximal ADMM (ADM-SPRO or ADM-IPRO) at the iterations, especially when the accurate estimation of maximum eigenvalues is difficult.
\end{itemize}
}

\section{Conclusions}\label{section:conclusions}
In this paper, we have proposed a special proximal ADMM where the proximal matrix derived from the BFGS \rev{update} or limited memory BFGS method.
\rev{We have given two remedies for the proximal matrix with the BFGS update to ensure the global convergence of such method.}
 Numerical results on several random problems with the large scale data \rev{have been} given to illustrate the effectiveness of the proposed method.

\rev{Recall that} Theorem \ref{theorem:basic} holds only when the Hessian matrix of the augmented Lagrangian function, that is, \rev{$M = \beta I + A^\top A$} is a constant matrix.
As a future work, we will consider more general problems by ADMM with \rev{the} BFGS update whose $x$-subproblems become unconstrained quadratic programming problem as in this paper. Then we may apply Theorem \ref{theorem:basic} for global convergence.
\rev{On the other hand, as shown in the numerical results, the ADMM with the L-BFGS also works well with a slight indefinite proximal matrix. This will facilitate the future exploration for an indefinite proximal ADMM with the BFGS update.}
\small

\bibliographystyle{siam}
\bibliography{LBFGS}

\clearpage
\begin{spacing}{1}
 \begin{appendices}
\section{Convergence of variable metric semi-proximal ADMM}\label{section:convergence property}
\rev{
Before showing the proof of Theorem \ref{theo:conv}, we give some notations and properties which will be frequently used in the analysis.
At last we provide the detail proof of Theorem \ref{theo:conv}.

Let $(x^*, y^*)$ be an optimal solution of problem \eqref{cp}, and let $\lambda^*$ be a Lagrange multiplier that satisfies the following KKT conditions of problem \eqref{cp}:
\vspace{-1mm}
\begin{subnumcases}{\label{kkt}}
  \eta_f^* - \la^* = 0,\\
  \eta_g^* + \la^* = 0,\\
 x^* - y^* =0,
\end{subnumcases}
 where $\eta_f^* \in \partial f(x^*)$ and $\eta_g^* \in \partial g(y^*).$
 }



%

\medskip
Now we rewrite the iteration \rev{schemes \eqref{equ:ADB1}-\eqref{equ:ADB2}}. Let \rev{$\Omega = \R^n \times \R^n \times \R^n.$}
Using the first-order optimality conditions \rev{for subproblems \eqref{equ:ADB1}-\eqref{equ:ADB2}, we see that} the new iterate $(x^{k+1}, y^{k+1})$ is generated by the following procedure.
 \begin{itemize}
{\item{step 1:}~~  Find $x^{k+1} \in \R^n$ such that $\rev{\eta_f^{k+1}} \in \partial f(x^{k+1})$ and
\begin{equation*}
 \rev{\eta_f^{k+1}} - \la^k + \beta(x^{k+1}-y^k) + T_k (x^{k+1} - x^k) = 0,
\end{equation*}
 \item{step 2:}~~  Find $y^{k+1} \in \R^n$ such that $\rev{\eta_g^{k+1}} \in \partial g(y^{k+1})$ and
\begin{equation*}
 \rev{\eta_g^{k+1}} + \la^k - \beta(x^{k+1}-y^{k+1}) + S (y^{k+1} - y^k) = 0.
\end{equation*}
}
\end{itemize}

\medskip
For $k = 0,1,2,...,$ we use the following notation:
\[\label{not}
u^* = \left(\begin{array}{c}x^* \\ y^* \end{array}\right),\; u^k = \left(\begin{array}{c}x^k \\ y^k \end{array}\right),\;  w^k = \left(\begin{array}{c}x^k\\ y^k\\ \lambda^k\end{array}\right), \;  D_k = \left(\begin{array}{c c}T_k & 0\\ 0 & S\end{array}\right), \; \mathrm{and} \; G_k = \left(\begin{array}{c c c} T_k & 0 & 0\\ 0 & S + \beta I & 0 \\ 0 & 0 & \frac{1}{\beta}I\end{array}\right).
\]
Moreover, for simplicity, we denote
\rev{
\[\label{def-f} F^k = \left(\begin{array}{c}
 \eta_f^k - \lambda^k \\ \eta_g^k+ \lambda^k \\ x^k - y^k \end{array} \right),\]
 where $\eta_f^k$ and $\eta_g^k$} are obtained in steps 1 and 2 in VMSP-ADMM.

\rev{For the sequences $\{w^k\}$ and $\{F^k\}$, we have the following lemma, which is a direct consequence of \cite[Theorem 1]{HLHY2002} and \cite[Lemma 3]{HLHY2002}.}
\begin{lemma}
\label{lemma:con}
Let $w^* = (x^*, y^*, \la^*)$, and $\{w^k\}$ be generated by the scheme \eqref{equ:ADMMB}. Then \rev{we have the following two statements.
\begin{description}
{\item[(i)] ~~ $\|w^{k+1} - w^*\|_{G_k}^2  \leq  \|w^k - w^*\|_{G_k}^2  - (\|u^{k+1} - u^k\|_{D_k}^2 + \beta \|x^{k+1} - y^k\|^2).$
 \item[(ii)] ~~ Suppose that sequence $\{T_k\}$ is bounded. Then, there exists a constant $\mu > 0$ such that for all $k\geq 0$, we have
\begin{equation*}
\|F^{k+1}\| \leq \mu \left( \|u^{k+1} - u^k \|_{D_k}^2 + \|x^{k+1} - y^k\|^2 \right).
\end{equation*}
}
\end{description}
}
\end{lemma}
\begin{proof}
The proofs of (i) and (ii) can be found in \cite[Theorem 1]{HLHY2002} and \cite[Lemma 3]{HLHY2002}, respectively.
\end{proof}

\rev{Now suppose that Condition \ref{cond} in Subsection \ref{subsection:convergence} holds.}
From the definitions of $\{D_k\}$ \rev{and $\{G_k\}$} in \eqref{not}, together with $T_k \succeq T \succeq 0, S \succeq 0$ and $\beta >0$, it follows that the sequences $\{D_k\}$ \rev{and $\{G_k\}$} also satisfy $0 \preceq D \preceq D_{k+1} \preceq (1+\gamma_k) D_k$, \rev{and $0 \preceq \bar{G} \preceq G_{k+1} \preceq (1+\gamma_k) G_k$} for all $k$, where $D = \left(\begin{array}{c c} T & 0\\ 0 & S\end{array}\right)\; \mathrm{and} \;
\bar{G} = \left(\begin{array}{c c c} T & 0 & 0\\ 0 & S + \beta I & 0 \\ 0 & 0 & \frac{1}{\beta}I\end{array}\right)$, respectively.

We define two constants $C_s$ and $C_p$ as follows:
\[\label{conditon:c}
C_s\colon = \sum_{k=0}^{\infty} \gamma_k  \; \; \mathrm{and} \;\;  C_p\colon = \prod_{k=0}^{\infty} (1+\gamma_k).
\]

\rev{Condition \ref{cond} (ii)} implies that  $0 \leq C_s < \infty$ and $1 \leq C_p < \infty$. \rev{Moreover, we have}
$ T \preceq T_{k} \preceq C_p T_0$ for all $k$,
which means that the sequences $\{T_k\}$ and $\{D_k\}$ are bounded.\par

\medskip
Now we give the proof of Theorem \ref{theo:conv}.

\begin{proof}[\normalsize{\emph{\textbf{Proof of Theorem \ref{theo:conv}:}}}]
First we show that the sequence $\{w^k\}$ is bounded.
\rev{Since $\bar{G} \preceq G_{k+1} \preceq (1+\gamma_k) G_k$,
we have}
\[\label{theo:2:2}
\|w^{k+1} - w^*\|_{G_{k+1}}^2 \leq (1+\gamma_k)\|w^{k+1} - w^*\|_{G_k}^2.
\]

Combining the inequality \eqref{theo:2:2} with \rev{Lemma \ref{lemma:con} (i)}, we have
\vspace{-1mm}
\begin{align}\label{theo:conv:1}
\|w^{k+1} - w^*\|_{G_{k+1}}^2 \leq{}& (1+\gamma_k)\|w^k - w^*\|_{G_k}^2 - (1+\gamma_k)\left( \|u^{k+1} - u^k\|_{D_k}^2 + \beta \|x^{k+1} - y^k\|^2\right) \nn \\
\leq{} & (1+\gamma_k)\|w^k - w^*\|_{G_k}^2 - c_1 \left( \|u^{k+1} - u^k\|_{D_k}^2 + \|x^{k+1} - y^k\|^2\right),
\end{align}
where $c_1 = \min \{1, \beta\}$.
It then follows that we have for all $k$,
\[\label{theo:conv:2}
\|w^{k+1} - w^*\|_{G_{k+1}}^2
\leq  (1+\gamma_k)\|w^k - w^*\|_{G_k}^2
\leq\cdots
\leq  \rev{\left(\prod_{i=0}^{k} (1+\gamma_i)\right)} \|w^0 - w^*\|_{G_0}^2
\leq C_p \|w^0 - w^*\|_{G_0}^2.
\]

\rev{Note that $\|w^{k+1} - w^*\|_{G_{k+1}}^2 = \|x^{k+1} - x^*\|_{T_{k+1}}^2 + \|y^{k+1} - y^*\|_{S+\beta I}^2 + \frac 1 \beta \|\lambda^{k+1} - \lambda^*\|^2$, $\S+\beta I$ is positive definite, and $C_p \|w^0 - w^*\|_{G_0}^2$ is a constant. It then follows from \eqref{theo:conv:2} that $\{y^k\}$ and $\{\la^k\}$ are bounded. We now show that $\{x^k\}$ is also bounded.}

From \eqref{theo:conv:1} and \eqref{theo:conv:2}, we have
\begin{align*}
c_1 \left( \|u^{k+1} - u^k\|_{D_k}^2 + \|x^{k+1} - y^k\|^2\right) \leq \|w^k - w^*\|_{G_k}^2 -  \|w^{k+1} - w^*\|_{G_{k+1}}^2 + \gamma_k  C_p \|w^0 - w^*\|_{G_0}^2.
\end{align*}
\rev{Summing up the inequalities,} we obtain
\begin{align*}
\sum_{k=0}^{\infty}c_1 \left( \|u^{k+1} - u^k\|_{D_k}^2 + \|x^{k+1} - y^k\|^2\right)
&{} \leq{} \|w^0 - w^*\|_{G_0}^2 -  \|w^{k+1} - w^*\|_{G_{k+1}}^2 + \left(\sum_{k=0}^{\infty}\gamma_k\right) C_p \|w^0 - w^*\|_{G_0}^2 \nn \\
&{} \leq{}  (1+ C_s C_p)\|w^0 - w^*\|_{G_0}^2.
\end{align*}
\rev{Since $(1+ C_s C_p)\|w^0 - w^*\|_{G_0}^2$ is a finite constant,} we have
\[\label{theo:conv:3}
\lim_{k\rightarrow \infty} \left( \|u^{k+1} - u^k\|_{D_k}^2 + \|x^{k+1} - y^k\|^2\right) = 0,
\]
which indicates that
\[\label{theo:lim} \lim_{k\rightarrow \infty} \|x^{k+1} - y^k\| = 0. \]

\rev{
Note that $x^* = y^*$ and $\|x^{k+1} - x^*\| = \|x^{k+1} - y^k + y^k - y^*\| \leq \|x^{k+1} - y^k\| + \|y^k - y^*\|$. It then follows from \eqref{theo:lim} that $\{x^k\}$ is bounded.
Consequently, the sequence $\{w^k\}$ is bounded.

Next we show that any cluster point of the sequence $\{w^k\}$ is a KKT point of \eqref{cp}.
}
Since the sequence $\{w^k\}$ is bounded, it has at least one cluster point in $\Omega$. Let \rev{$w^ \infty = (x^\infty, y^\infty, \lambda^\infty) \in \Omega$} be a cluster point of $\{w^k\}$, and let $\{w^{k_j}\}$ be a subsequence of $\{w^k\}$ that converges to point $w^ \infty$.

\rev{From \eqref{theo:conv:3} and Lemma \ref{lemma:con} (ii), we have
$\lim_{j\rightarrow \infty} \|F^{k_j}\| = 0$.
It then follows from the definition of $F^k$ that $x^\infty = y^\infty$.
Moreover, since $\partial f$ and $\partial g$ are upper semi-continuous,
there exists $\eta_f^\infty$ and $\eta_g^\infty$ such that $\eta_f^\infty \in \partial f(x^\infty)$, $\eta_g^\infty \in \partial g(y^\infty)$, $\eta_f^{k_j} \rightarrow \eta_f^\infty$ and $\eta_g^{k_j} \rightarrow \eta_g^\infty$, taking a subsequence if necessary. It then follows from $\lim_{j\rightarrow \infty} \|F^{k_j}\| = 0$ that $\eta_f^\infty - \lambda^\infty = 0$ and $\eta_g^\infty + \lambda^\infty = 0$. Consequently $w^\infty$ satisfies the KKT conditions of problem \eqref{cp}.

Finally, we show that the whole sequence $\{w^k\}$ converges to $w^\infty$.
}

Since $\{w^{k_j}\}$ converges to $w^\infty$, for any positive scalar $\epsilon$, there exists positive integer $q$ such that
\[\label{theo:conv:6}
\|w^{k_q} - w^\infty\|_{G_{k_q}} < \frac{\epsilon}{\rev{C_p^{\frac 1 2}}},
\]
\rev{
Note that \eqref{theo:conv:2} holds for an arbitrary KKT point $w^*$ of problem \eqref{cp}. It then follows from \eqref{theo:conv:2} with $w^* = w^\infty$ that for any $k \geq k_q$, we have
\vspace{-1mm}
\begin{align*}
\|w^k - w^\infty\|_{G_{k}} \leq{} \left(\prod_{i=k_q}^{k-1} (1+\gamma_i)\right)^{1/2} \|w^{k_q} - w^\infty\|_{G_{k_q}}
\leq{}  \left(\prod_{i=0}^{\infty} (1+\gamma_i)\right)^{1/2} \|w^{k_q} - w^\infty\|_{G_{k_q}}
< \epsilon,
\end{align*}
where the second inequality follows from Condition \ref{cond} (ii), and the last inequality follows from \eqref{theo:conv:6} and the definition of $C_p$.
Since $\epsilon$ is an arbitrary positive scalar, this shows that $\{w^k\}$ converges to $w^\infty$.
}
\end{proof}

\end{appendices}
\end{spacing}

\end{document}